\documentclass[11pt]{amsart}
\usepackage{graphicx,subcaption}
\usepackage{amsmath} 
\usepackage{amsthm,amsfonts,amssymb,mathrsfs,amscd,amstext,amsbsy} 
\usepackage{epic,eepic} 
\usepackage{yfonts}
\usepackage{paralist,enumerate}
\usepackage[all]{xy}
\usepackage{hyperref}
\usepackage{tikz}
\usetikzlibrary {positioning}
\hypersetup{colorlinks}

\newtheorem{theorem}{Theorem}[section]
\newtheorem{cor}[theorem]{Corollary}
\newtheorem{lemma}[theorem]{Lemma}

\newtheorem{theo}[theorem]{Theorem}
\newtheorem{lem}[theorem]{Lemma}
\newtheorem{pro}[theorem]{Proposition}
\newtheorem{rem}[theorem]{Remark}
\newtheorem{exa}[theorem]{Example}
\newtheorem{que}[theorem]{Question}
\newtheorem{Definition}[theorem]{Definition}
\newtheorem*{Definition*}{Definition}
\numberwithin{equation}{section}
\def\qed{\hfill \ifhmode\unskip\nobreak\fi\quad\ifmmode\Box\else$\Box$\fi\\ }
\begin{document}
\title[$S^1$-action on oriented manifold with discrete fixed set]{Circle actions on oriented manifolds with discrete fixed point sets and classification in dimension 4}
\author{Donghoon Jang}
\thanks{Donghoon Jang is supported by Basic Science Research Program through the National Research Foundation of Korea(NRF) funded by the Ministry of Education(2018R1D1A1B07049511).}
\address{Department of Mathematics, Pusan National University, Pusan, Korea}
\email{donghoonjang@pusan.ac.kr}
\begin{abstract}
In this paper, we study a circle action on a compact oriented manifold with a discrete fixed point set. The fixed point data consists of the weights of the $S^1$-representations at the fixed points. We prove various results and properties of the action, in terms of the fixed point data. We show that the manifold can be described by a multigraph associated to it.

Specializing into the case of dimension 4, we classify the fixed point data. Moreover, we prove that there exist oriented $S^1$-manifolds with these fixed point data. Finally, we show that a certain multigraph behaves like a manifold.
\end{abstract}
\maketitle
\tableofcontents
\section{Introduction}

In this paper, we study circle actions on oriented manifolds with discrete fixed point sets. Let the circle act on a compact oriented manifold $M$ with a discrete fixed point set. At each fixed point, there are non-zero integers, called \emph{weights} (also called \emph{rotation numbers}). In this paper, we prove properties of the weights at the fixed points and derive results on the manifold. One of which we show is that we can associate a multigraph to $M$ and the manifold can be described by the multigraph. Finally, we specialize into the case of dimension 4. In dimension 4, we classify the weights at the fixed points and prove the existence part. Moreover, we give a necessary and sufficient condition for a multigraph to be realized as a multigraph associated to a 4-dimensional oriented $S^1$-manifold.

Consider a circle action on a compact oriented manifold. Assume that the fixed point set is non-empty and finite. For the classification of such an action, one may want to begin either with small numbers of fixed points, or with low dimensions. Note that having an isolated fixed point implies that the dimension of the manifold is even.

First, let us begin with small numbers of fixed points. If there is one fixed point, then the manifold must be a point. On the other hand, if there are two fixed points, then any even dimension is possible, as there is an example of a rotation of an even dimensional sphere with two fixed points. From the example, on any even dimension greater than two, we can have any even number of fixed points, since we can perform equivariant sum of rotations of even dimensional spheres. This makes a big difference with $S^1$-actions on other types of manifolds; for instance, an almost complex (and hence complex or symplectic) manifold $M$ equipped with an $S^1$-action\footnote{Throughout the paper, if the circle acts on an almost complex, complex, or symplectic manifold, we assume that the action preserves the almost complex, complex, or symplectic structure, respectively.} having two fixed points must have either $\dim M=2$ or $\dim M=6$; for the classification results for other types of manifolds, see \cite{Jan2}, \cite{Kos1}, and \cite{PT}.

The situation is quite different when the number of fixed points is odd. If there is an odd number of fixed points, then the dimension of the manifold must be a multiple of four; see Corollary \ref{c27}. Let us specialize into the case of three fixed points. Then the complex, quaternionic, and octonionic (Cayley) projective spaces ($\mathbb{CP}^2$, $\mathbb{HP}^2$, and $\mathbb{OP}^2$) of dimension 2 admit circle actions with three fixed points, which have real dimensions 4, 8, and 16, respectively. On the other hand, to the author's knowledge, it is not known if in dimensions other than 4, 8, and 16, there exists a manifold with three fixed points. Similar to the case of two fixed points, if we assume an almost complex structure on a manifold, three fixed points can only happen in dimension 4 \cite{Jan2}. Note that among the spaces above, only $\mathbb{CP}^2$ admits an almost complex structure (and complex or symplectic structures). For the classification results for other types of manifolds, see \cite{Jan1}.

Second, let us begin with low dimensions. In dimension two, the classification is rather trivial. Among compact oriented surfaces, only the 2-sphere $S^2$ and the 2-torus $\mathbb{T}^2$ admit non-trivial circle actions. Any circle action on $S^2$ has two fixed points and any circle action on $\mathbb{T}^2$ is fixed point free; see Lemma \ref{l212}.

We discuss classification results in dimension four. Before we discuss our main result, let us discuss results for other types of manifolds, or related results. The classification of a holomorphic vector field on a complex surface is made by Carrell, Howard, and Kosniowski \cite{CHK}. For a 4-dimensional Hamiltonian $S^1$-space, subsequent to the work by Ahara and Hattori \cite{AH} and Audin \cite{Au}, Karshon classifies such a space up to equivariant symplectomorphism, in terms of a multigraph associated to $M$ \cite{Ka}. Note that in Section \ref{s4}, we shall associate a multigraph to an oriented manifold $M$ and our notion of multigraphs generalizes the multigraphs for 4-dimensional Hamiltonian $S^1$-spaces. A multigraph determines the weights at fixed points. In dimension 4, for a complex manifold or for a symplectic manifold, the weights at the fixed points determine the manifold uniquely. When the fixed point set are discrete, our main result generalizes the classification of weights at fixed points for complex manifolds and symplectic manifolds to oriented manifolds. However, for a manifold to be oriented is a very weak condition, and therefore uniqueness fails to hold for oriented manifolds, since we can perform equivariant sum of a manifold with another manifold that is fixed point free.

Somewhat related results are the classifications of a circle action on a 4-manifold, with different perspectives. For instance, a circle action on a homotopy 4-sphere \cite{MY1}, \cite{MY2}, \cite{F1}, \cite{Pa} or on a simply connected 4-manifold \cite{F2}, \cite{Y} is considered. In addition, Fintushel classifies a 4-dimensional oriented $S^1$-manifold in terms of orbit data \cite{F3}.

Third, there is another point of view that one considers for a circle action on an oriented manifold with a discrete fixed point set. One of them is the Petrie's conjecture, which asserts that if a homotopy $\mathbb{CP}^n$ admits a non-trivial $S^1$-action, then its total Pontryagin class is the same as that of $\mathbb{CP}^n$ \cite{P}. In other words, the existence of a non-trivial $S^1$-action is enough to determine the characteristic class of such a manifold. The Petrie's conjecture is proved to hold in dimension up to 8 \cite{D}, \cite{Ja}.

To state our classification result, we introduce a terminology. Let the circle act on a $2n$-dimensional compact oriented manifold $M$ with a discrete fixed point set. Let $p$ be a fixed point. Then the tangent space at $p$ decomposes into $n$ two-dimensional irreducible $S^1$-equivariant real vector spaces
\begin{center}
$T_pM=\bigoplus_{i=1}^n L_i$.
\end{center}
Each $L_i$ is isomorphic to a one-dimensional $S^1$-equivariant complex space on which the action is given as multiplication by $g^{w_p^i}$, where $g\in S^1$ and $w_p^i$ is a non-zero integer. The $w_p^i$ are called \textbf{weights} at $p$. Though the sign of each weight is not well-defined, the sign of the product of the weights at $p$ is well-defined. We orient each $L_i$ so that every weight is positive. Let $\epsilon(p)=+1$ if the orientation given on $\bigoplus_{i=1}^n L_i$ this way agrees on the orientation on $T_pM$ and $\epsilon(p)=-1$ otherwise. Let us call it the \textbf{sign} of $p$. Denote the fixed point data at $p$ by $\Sigma_p=\{\epsilon(p), w_p^1, \cdots, w_p^n\}$. By the fixed point data $\Sigma_M$ of $M$, we mean a collection $\displaystyle \cup_{p \in M^{S^1}} \Sigma_p$ of the fixed point data at each fixed point $p$. To avoid possible confusion with weights, when we write the sign at $p$ inside $\Sigma_p$, we shall only write the sign of $\epsilon(p)$ and omit 1.

We give an example. Let the circle act on $S^{2n}$ by 
\begin{center}$g \cdot (z_1,\cdots,z_n,x)=(g^{a_1} z_1,\cdots, g^{a_n} z_n, x)$, \end{center}
for any $g \in S^1 \subset \mathbb{C}$, where $z_i$ are complex numbers and $x$ is a real number such that $\sum_{i=1}^n |z_i|^2+x^2=1$, and $a_i$ are positive integers for $1 \leq i \leq n$. The action has two fixed points, $p=(0,\cdots,0,1)$ and $q=(0,\cdots,0,-1)$. Near $p$, the action is described as $g \cdot (z_1,\cdots,z_n)=(g^{a_1} z_1,\cdots, g^{a_n} z_n)$. Therefore, the weights at $p$ are $\{a_1,\cdots,a_n\}$. Similarly, the weights at $q$ are $\{a_1,\cdots,a_n\}$. On the other hand, we have that $\epsilon(p)=-\epsilon(q)$; see Theorem \ref{t29}. The fixed point data of the circle action on $S^{2n}$ is therefore $\{+,a_1,\cdots,a_n\} \cup \{-,a_1,\cdots,a_n\}$.

With the notion of weights, consider a circle action on a 4-dimensional compact oriented manifold $M$ and assume that the fixed point set is discrete. As we have seen, the classification of the fixed point data for oriented manifolds is in general harder than complex manifolds or symplectic manifolds. To the author's knowledge, the fixed point data of an $S^1$-action on an oriented 4-manifold is known only if the number of fixed points is at most three; see \cite{L2} for the case of three fixed points. In this paper, we completely determine the fixed point data of $M$ with an arbitrary number of fixed points. Note that given a circle action on a manifold, we can always make the action effective by quotienting out by the subgroup $\mathbb{Z}_k$ that acts trivially. This amounts to dividing all the weights by $k$. We prove that for a circle action on a 4-dimensional oriented manifold with a discrete fixed point set, the fixed point data of the manifold can be achieved by simple combinatorics. A combinatorial format of the main result can be stated as follows.

\begin{theorem} \label{t11} Let the circle act effectively on a 4-dimensional compact oriented manifold $M$ with a discrete fixed point set. Then the fixed point data of $M$ can be achieved in the following way: begin with the empty set, and apply a combination of the following steps.
\begin{enumerate}
\item Add $\{+,a,b\}$ and $\{-,a,b\}$, where $a$ and $b$ are relatively prime positive integers.
\item Replace $\{+,c,d\}$ by $\{+,c,c+d\}$ and $\{+,d,c+d\}$.
\item Replace $\{-,e,f\}$ by $\{-,e,e+f\}$ and $\{-,f,e+f\}$.
\end{enumerate}\end{theorem}

For instance, if there are 2 fixed points, the only possibility is that only Step (1) occurs once, and hence the fixed point data is $\{+,a,b\}$ and $\{-,a,b\}$ for some positive integers $a$ and $b$. If there are 3 fixed points, Step (1) needs to occur once, and exactly one of Step (2) and Step (3) must occur once; if Step (2) occurs the fixed point data is $\{+,a,a+b\} \cup \{+,b,a+b\} \cup \{-,a,b\}$ and if Step (3) occurs the fixed point data is $\{+,a,b\} \cup \{-,a,a+b\} \cup \{-,b,a+b\}$ for some positive integers $a$ and $b$; see Theorem \ref{t71}.

Moreover, we prove the existence part. Given a fixed point data achieved by the combinatorics as in Theorem \ref{t11}, we construct a manifold equipped with a circle action having a discrete fixed point set, with the desired fixed point data. In our construction, Step (1) in Theorem \ref{t11} corresponds to adding $S^4$ equipped with a rotation, and Step (2) and Step (3) each correspond to blowing up a fixed point. A geometric format of the main result, that proves the existence of such a manifold $M$ in Theorem \ref{t11}, can be stated as follows.

\begin{theorem} \label{t12} Let the circle act on a 4-dimensional compact connected oriented manifold $M$ with a discrete fixed point set. Then there exists a 4-dimensional compact connected oriented manifold $M'$ that is an equivariant sum of blow-ups of rotations on $S^4$'s, such that $M$ and $M'$ have the same fixed point data. 

More precisely, $M'$ is an equivariant sum along free orbits of blow-ups of rotations of $S^4$'s, where $S^1$ acts on each $S^4$ by $g \cdot (z_1,z_2,x)=(g^az_1,g^bz_2,x)$ for any $g \in S^1 \subset \mathbb{C}$, for some positive integers $a$ and $b$, and blow up is in the sense of Lemma \ref{l51}. \end{theorem}

We adapt the notion of blow up because we shall identify a neighborhood of a fixed point in $M$ with a neighborhood of $0$ in $\mathbb{C}^2$, where we have a complex structure to be able to blow up in the usual sense. We shall blow up equivariantly and only do so at a fixed point. Suppose that we blow up a fixed point $p$ whose weights are $\{\pm,a,b\}$ for some positive integers $a$ and $b$. On a blown up manifold $\widetilde{M}$, there is a naturally extended circle action. Moreover, instead of the fixed point $p$, there are two fixed points $p_1$ and $p_2$, whose weights are $\{\pm,a,a+b\}$ and $\{\pm,b,a+b\}$, respectively. For details, see Section \ref{s5}.

As a corollary of Theorem \ref{t11} and Theorem \ref{t12}, we classify the signature of a 4-dimensional compact oriented manifold equipped with a circle action having a discrete fixed point set. We give a proof in Section \ref{s7}.

\begin{cor} \label{c13} Let the circle act on a 4-dimensional compact oriented manifold $M$ with $k$ fixed points. Then $k \geq 2$, and the signature of $M$ satisfies $2-k \leq \textrm{sign}(M) \leq k-2$ and $\textrm{sign}(M) \equiv k \mod 2$. Moreover, given any pair $(j,k)$ of integers $j$ and $k$ such that $k \geq 2$, $2-k \leq j \leq k-2$, and $j \equiv k \mod 2$, there exists a 4-dimensional compact connected oriented manifold $M$ equipped with a circle action having $k$ fixed points, whose signature satisfies $\textrm{sign}(M)=j$. \end{cor}

The idea of the proof of Theorem \ref{t11} and Theorem \ref{t12} is as follows. Given a weight $w$, there exist two fixed points $p_1$ and $p_2$ that lie in the same connected component $Z$ of $M^{\mathbb{Z}_w}$, where $M^{\mathbb{Z}_w}$ denotes the set of points fixed by the $\mathbb{Z}_w$-action. Here $\mathbb{Z}_w$ acts on $M$ as a subgroup of $S^1$. We begin with the biggest weight. Suppose that $p_1$ and $p_2$ have the biggest weight. Then we prove that two cases occur. In one case, we show that the fixed point data of $M$ can be obtained from another manifold $M'$ by blowing up at a fixed point, and in the other case we show that the fixed point data of $M$ can be obtained from an equivariant sum of another manifold $M''$ with $S^4$. In the former case, the fixed point data at $p_1$ and $p_2$ is achieved by blowing up a fixed point in $M'$. In the latter case, a rotation of $S^4$ has fixed points whose fixed point data is the same as that of $p_1$ and $p_2$. Therefore, these manifolds $M'$ and $M''$ have fewer fixed points and have smaller largest weights. Now, on $M'$ or $M''$, pick the biggest weight. The classification problem now reduces to the existence of another manifold with fewer fixed points and with smaller largest weight. We continue this process to reduce the classification problem of the fixed point data to the existence of a semi-free $S^1$-action with a discrete fixed point set, that does exist.

Let us discuss what kinds of invariants can be determined from Theorem \ref{t11}. For this, suppose that there exists a 4-dimensional compact oriented manifold $M$ equipped with a circle action having a discrete fixed point set. Then we can perform equivariant sum of $M$ with another manifold $M'$ that is fixed point free. The resulting manifold is in general not (equivariantly) diffeomorphic to $M$. Therefore, Theorem \ref{t11} cannot determine whether two manifolds with the same fixed point data are (equivariantly) diffeomorphic to each other or not.

On the other hand, Theorem \ref{t11} determines some invariants of a given manifold $M$. Theorem \ref{t11} determines the signature, the Pontryagin number $\int_M p_1$, and the Euler characteristic of $M$. In particular, let $b_1$ and $b_2$ be the number of times Step (2) and Step (3) occur in Theorem \ref{t11}, respectively. Then the signature of $M$ is $b_1-b_2$. The signature of any circle action on a 4-dimensional almost complex manifold satisfies $|\textrm{sign}(M)| \leq |k-4|$ ($|\textrm{sign}(M)|=|k-4|$ if $M$ is complex or symplectic), where $k$ is the number of fixed points. We can see this in the following way: if $M$ is almost complex, then the sign of each weight at a fixed point is well-defined. For $0 \leq i \leq 2$, there exists at least one fixed point with $i$ negative weights. A fixed point with $i$ negative weights contributes as $(-1)^i$ to the signature of $M$. On the other hand, if $M$ is complex or symplectic, then there is a unique fixed point with no negative weight and there is a unique fixed point with 2 negative weights. However, Corollary \ref{c13} proves that there exists a 4-dimensional oriented manifold equipped with a circle action having $k$ fixed points, whose signature satisfies $|\textrm{sign}(M)|=k-2$.

A natural question from Theorem \ref{t11} is to ask if an analogous result holds in higher dimensions.
\begin{que} \label{q1} Let the circle act on a compact oriented manifold with a discrete fixed point set. Then can we classify the fixed point data? What is a possible fixed point data? What can we say about the manifold? \end{que}
More precisely, it is natural to ask if there exist minimal models (manifolds) and operations for classifying fixed point data.
\begin{que} \label{q2} Let the circle act on a compact oriented manifold $M$ with a discrete fixed point set. Then what are the minimal models (manifolds) and operations to construct a manifold $M'$ with the same fixed point data as $M$? \end{que}
As in dimension 4, spheres, equivariant sum, and blow-up operation are possible candidates for the problem.

This paper is organized as follows. For the description, let the circle act on a compact oriented manifold $M$ with a discrete fixed point set. In Section \ref{s2}, we recall background and prove properties of weights and derive results from the properties. In Section \ref{s3}, we prove properties that weight representations of $M$ at the fixed points satisfy, in terms of isotropy submanifolds. One of which, Lemma \ref{l34}, plays a crucial role in the proof of Theorem \ref{t11} and Theorem \ref{t12}. In Section \ref{s4}, we associate a multigraph to $M$. In Section \ref{s5}, we discuss the blow up operation that also plays a key role in the proof. In Section \ref{s6}, with all of these, we prove Theorem \ref{t11} and Theorem \ref{t12} together. In Section \ref{s7}, we classify a circle action with few fixed points and prove Corollary \ref{c13} as applications of Theorem \ref{t11}. In Section \ref{s8}, as another application, we show that some kind of multigraph behaves like a manifold.

The motivation of this paper is \cite{L2}, where Li classifies the weights at the fixed points of a circle action on a 4-dimensional compact orientable manifold with 3 fixed points. The author tried the case of 4 fixed points and found out that there is some pattern. Therefore the author developed a theory to deal with any number of fixed points (and any dimension), and was able to classify the case of any number of fixed points. After writing this paper, the author was informed that circle actions on 4-dimensional manifolds were classified a long time ago (see \cite{F1}, \cite{F2}, \cite{F3}, and \cite{Y}), and Theorem \ref{t11} of this paper can be extracted by some combinatorics of those papers; the weights at any fixed point are encoded in the Seifert invariants of the orbit space (see \cite{F3}). The author was also informed after writing this paper that the author's combinatorial technique (see the proof of the main theorems in Section \ref{s6}) was used by Pao \cite{Pa}, where Pao only considers homotopy 4-spheres. However, our proof is different from those papers (their main tool is orbit data), and those results and techniques are restricted in dimension 4. In this paper, we develop a theory that can be used to study in arbitrary dimension; see Sections \ref{s2}, \ref{s3}, \ref{s4}, and \ref{s5}. The author's combinatorial method follows from Lemma \ref{l34}, that we obtain as a corollary of Lemma \ref{l31} which is for arbitrary dimension. Similarly, the association of a multigraph to a manifold (Section \ref{s4}) can be used in any dimension.

Finally, the author would like to thank the anonymous referee for valuable comments and suggestions.

\section{Circle action on oriented manifold with discrete fixed point set} \label{s2}

Let the circle act on a compact oriented manifold with a discrete fixed point set. In this section, we recall background and prove properties that the weights at the fixed points satisfy and look at their bi-products.

For a compact oriented manifold $M$, one defines the signature operator on $M$. The Atiyah-Singer index theorem states that the topological index of the operator is equal to the analytical index of the operator. With the definitions of weights and the sign at fixed points in the Introduction, as an application to a compact oriented $S^1$-manifold with a discrete fixed point set, we have the following formula:

\begin{theo} \emph{[Atiyah-Singer index theorem]} \cite{AS} \label{t21} Let the circle act on a $2n$-dimensional compact oriented manifold $M$ with a discrete fixed point set. Then the signature of $M$ is
\begin{center}
$\displaystyle{\textrm{sign}(M) = \sum_{p \in M^{S^1}} \epsilon(p) \prod_{i=1}^{n} \frac{(1+t^{w_p^i})}{(1-t^{w_p^i})}}$
\end{center}
and is a constant, where $t$ is an indeterminate. \end{theo}
In particular, taking $t=0$ in Theorem \ref{t21}, we have
\begin{equation} \label{eq:1}
\displaystyle{\textrm{sign}(M) = \sum_{p \in M^{S^1}} \epsilon(p).}
\end{equation}

One of the important operations to prove Theorem \ref{t11} is the equivariant sum that connects two $S^1$-manifolds equivariantly. Let us discuss on equivariant sum more precisely. For $i=1,2$, let $M_i$ be a $2n$-dimensional compact connected oriented manifold equipped with an effective circle action and with a discrete fixed point set. Then for each $i$, there exists a free orbit in $M_i$. For each $i$, by the slice theorem, there exists a tubular neighborhood $S^1 \times D^{2n-1}$ in $M_i$, containing the free orbit. Using the tubular neighborhoods, we can connect $M_1$ and $M_2$ equivariantly to construct a new manifold $M$, that is compact, connected, oriented, and is equipped with a circle action. Moreover, by reversing orientations of $M_i$ and gluing if necessary, $M$ has the fixed point data $\pm \Sigma_{M_1} \cup \pm \Sigma_{M_2}$, where $-\Sigma_{M_i}$ denotes the fixed point data of $M_i$ with the orientation reversed, i.e., if $\Sigma_{M_i}=\cup_{p \in M_i^{S^1}} \{\epsilon(p),w_p^1,\cdots,w_p^n\}$, then $-\Sigma_{M_i}=\cup_{p \in M_i^{S^1}} \{-\epsilon(p),w_p^1,\cdots,w_p^n\}$. For the complete proof, one may look at Lemma 2.2 and Proposition 2.4 of \cite{L2}.

\begin{lem} \cite{L2} \label{l23} Let the circle act effectively on a $2n$-dimensional compact connected oriented manifold $M_i$ with a discrete fixed point set, for $i=1,2$, where $n>1$. Then we can construct a $2n$-dimensional compact connected oriented manifold $M$ equipped with a circle action, whose fixed point data is $\pm \Sigma_{M_1} \cup \pm \Sigma_{M_2}$. \end{lem}

In \cite{K}, Kobayashi shows that the fixed point set of an $S^1$-action on a compact orientable manifold is also orientable.

\begin{lem} \cite{K} \label{l24} Let the circle act on a compact orientable manifold $M$. Then its fixed point set is orientable. \end{lem}

Let $w>1$ be a positive integer. Given an effective circle action on a compact oriented manifold $M$, the group $\mathbb{Z}_w$ acts on $M$, as a subgroup of $S^1$, and the set $M^{\mathbb{Z}_w}$ of points that are fixed by the $\mathbb{Z}_w$-action is a union of smaller dimensional submanifolds. In \cite{HH}, H. Herrera and R. Herrera prove a result on the orientability of $M^{\mathbb{Z}_w}$.

\begin{lem} \label{l25} \cite{HH} Let the circle act effectively on a $2n$-dimensional compact oriented manifold. Consider $\mathbb{Z}_w \subset S^1$ and its corresponding action on $M$. If $w$ is odd then the fixed point set $M^{\mathbb{Z}_w}$ is orientable. If $w$ is even and a connected component $Z$ of $M^{\mathbb{Z}_w}$ contains a fixed point of the $S^1$-action, then $Z$ is orientable. \end{lem}

We shall discuss consequences of Theorem \ref{t21}.

\begin{lem} \label{l22} Let the circle act on a $2n$-dimensional compact oriented manifold $M$ with a discrete fixed point set. Let $w$ be the smallest weight that occurs among all the fixed points. Then the number of times the weight $w$ occurs at the fixed points of sign $+1$, counted with multiplicity, is equal to the number of times the weight $w$ occurs at the fixed points of sign $-1$, counted with multiplicity. \end{lem}

\begin{proof}
By Theorem \ref{t21}, the signature of $M$ is
\begin{center}
$\displaystyle{\textrm{sign}(M) = \sum_{p \in M^{S^1}} \epsilon(p) \prod_{i=1}^{n} \frac{ (1+t^{w_p^i})}{(1-t^{w_p^i})} = \sum_{p \in M^{S^1}} \epsilon(p) \prod_{i=1}^{n} [ (1+t^{w_p^i}) ( \sum_{j=0}^{\infty} t^{j w_p^i} )].}$
\end{center}
Let $w$ be the smallest (positive) weight. At each fixed point $p$,
\begin{center}
$\displaystyle \prod_{i=1}^{n} [(1+t^{w_p^i}) ( \sum_{j=0}^{\infty} t^{j w_p^i} )] = 1+ 2N_p(w)t^w + t^{w+1}f_p(t)$,
\end{center}
where $N_p(w)$ is the number of times the weight $w$ occurs at $p$ and $f_p(t)$ is an infinite polynomial. Therefore,
\begin{center}
$\displaystyle \textrm{sign}(M) = \sum_{p \in M^{S^1}} \epsilon(p) [1+ 2N_p(w)t^w + t^{w+1}f_p(t)]$.
\end{center}
The signature of $M$ is independent of the indeterminate $t$ and is a constant. Collecting $t^w$-terms, we have
\begin{center}
$\sum_{p \in M^{S^1}} \epsilon(p) N_p(w)=0$. \end{center} \end{proof}

With Lemma \ref{l25}, the following lemma is obtained as an application of Lemma \ref{l22}.

\begin{lem} \label{l26} Let the circle act on a compact oriented manifold with a discrete fixed point set. For each positive integer $w$, the number of times $w$ occurs as weights among all the fixed points, counted with multiplicity, is even. \end{lem}

\begin{proof} Consider the set $M^{\mathbb{Z}_w}$ of points that are fixed by the $\mathbb{Z}_w$-action, where $\mathbb{Z}_w$ acts on $M$ as a subgroup of $S^1$. Let $Z$ be a connected component of $M^{\mathbb{Z}_w}$ that contains an $S^1$-fixed point. By Lemma \ref{l25}, $Z$ is orientable. Fix an orientation of $Z$. The circle action on $M$ restricts to a circle action on $Z$. The circle action on $Z$ has $Z \cap M^{S^1}$ as a fixed point set. The smallest weight of the $S^1$-action on $Z$ is $w$. By applying Lemma \ref{l22} to the $S^1$-action on $Z$, the number of times the weight $w$ occurs at the fixed points of $Z \cap M^{S^1}$ is even. \end{proof}

An immediate consequence of Lemma \ref{l26} is that if there is an odd number of fixed points, then the dimension of the manifold is a multiple of 4.

\begin{cor} \label{c27} Let the circle act on a compact oriented manifold $M$. If the number of fixed points is odd, the dimension of the manifold is divisible by four. \end{cor}

\begin{proof} Assume on the contrary that $\dim M=2n$ is not divisible by four. Then $n$ is odd. Let $k$ be the number of fixed points. Then the total number $nk$ of weights among all the fixed points, counted with multiplicity is odd. On the other hand, by Lemma \ref{l26}, for each positive integer $w$, the number of times $w$ occurs as weights among all the fixed points, counted with multiplicity is even. In particular, the total number $nk$ of weights among all the fixed points counted with multiplicity must be even, which leads to a contradiction. \end{proof}

If there are two fixed points, then the fixed point data is classified.

\begin{theorem} \label{t28} \cite{Kos2}, \cite{L2}
Let the circle act on a compact oriented manifold with two fixed points $p$ and $q$. Then the weights at $p$ and $q$ are equal and $\epsilon(p)=-\epsilon(q)$. \end{theorem}

The following result concerns the vanishing of the signature of a manifold equipped with a certain type of a circle action. This will be used in the proof of Theorem \ref{t11}.

\begin{theorem} \label{t29}
Let the circle act on a $2n$-dimensional compact oriented manifold $M$ with a discrete fixed point set. Suppose that the weights at each fixed point are $\{a_1,\cdots,a_n\}$ for some positive integers $a_1,\cdots,a_n$. Then the number of fixed points $p$ with $\epsilon(p)=+1$ and that with $\epsilon(p)=-1$ are equal. In particular, the signature of $M$ vanishes.
\end{theorem}

\begin{proof} By Theorem \ref{t21}, the signature of $M$ is
\begin{center}
$\displaystyle{\textrm{sign}(M) = \sum_{p \in M^{S^1}} \epsilon(p) \prod_{i=1}^{n} \frac{(1+t^{w_p^i})}{(1-t^{w_p^i})}}=\sum_{p \in M^{S^1}} \epsilon(p) \prod_{i=1}^{n} \frac{(1+t^{a_i})}{(1-t^{a_i})}$.
\end{center}
Since the signature of $M$ is a constant, we must have that $\textrm{sign}(M)=0$, and the number of fixed points $p$ with $\epsilon(p)=+1$ and that with $\epsilon(p)=-1$ are equal. \end{proof}

An $S^1$-action is called semi-free, if the action is free outside the fixed point set. In \cite{TW}, Tolman and Weitsman classify a symplectic semi-free circle action with a discrete fixed point set. Li reproves this in \cite{L1}. The case of almost complex manifolds is dealt in \cite{Jan3}. For a semi-free circle action, all the weights at the fixed points are 1. We classify a semi-free action on an oriented manifold with a discrete fixed point set, and solve the existence issue.

\begin{theorem} \label{t210} Let the circle act semi-freely on a compact oriented manifold $M$ with a discrete fixed point set. Then there is an even number of fixed points, and the signature of $M$ vanishes. Moreover, given any positive integer $k$, there exists a semi-free circle action on a compact oriented manifold $M$ with $2k$ fixed points. \end{theorem}

\begin{proof} In Theorem \ref{t29}, take $a_i=1$ for all $i$. The number of fixed points $p$ with $\epsilon(p)=+1$ and that with $\epsilon(p)=-1$ are equal. In particular there is an even number of fixed points. Moreover, the signature of $M$ vanishes.

For the latter part, consider a rotation of $S^{2n}$ as in the introduction with taking $a_i=1$ for all $i$. It has two fixed points, the north pole and the south pole, and its fixed point data is $\{+,1,\cdots,1\}$ and $\{-,1,\cdots,1\}$; the action is semi-free. Take $k$-copies of such a manifold, and take an equivariant sum along free orbits of the manifolds in the sense of Lemma \ref{l22}. \end{proof}

We quickly review the classification of $S^1$-actions on compact oriented surfaces. This will be used in the proof of Theorem \ref{t11}. Given a manifold $M$, let $\chi(M)$ be the Euler number of $M$.

\begin{theo} \cite{K} \label{t211} Let the circle act on a compact oriented manifold $M$. Then \begin{center} $\displaystyle \chi(M)=\sum_{Z \subset M^{S^1}} \chi(Z)$. \end{center} \end{theo}

The Euler number of a compact oriented surface $M$ is $2-2g$, where $g$ is the number of genus of $M$. The Euler number of a point is 1. Therefore, the following lemma holds.

\begin{lem} \label{l212}
Let $M$ be a compact connected oriented surface of genus $g$.
\begin{enumerate}
\item If $g=0$, i.e., $M$ is the 2-sphere $S^2$, then any circle action on it has two fixed points.
\item If $g=1$, i.e., $M$ is the 2-torus $\mathbb{T}^2$, then any circle action on it is fixed point free.
\item If $g>1$, then $M$ does not admit a non-trivial circle action.
\end{enumerate}
\end{lem}

\begin{proof} Suppose that $M$ admits a circle action. Since $\dim M=2$, the fixed point set is either the empty set or a finite number of points. By Theorem \ref{t211},
\begin{center} $\displaystyle \chi(M)=\sum_{p \in M^{S^1}} 1$, \end{center}
since the Euler number of a point is 1. This implies that $\chi(M) \geq 0$. On the other hand, $\chi(M)=2-2g$. This implies that $g=0$ or 1. If $g=0$, then $M$ is the 2-sphere and it has two fixed points since $2=\chi(M)=\sum_{p \in M^{S^1}} 1$. If $g=1$, then $M$ is the 2-torus and it has no fixed points since $0=\chi(M)=\sum_{p \in M^{S^1}} 1$. \end{proof}

\section{Weight representations} \label{s3}

In this section, we investigate properties that the weights at the fixed points satisfy, in terms of isotropy submanifolds. Our main goal in this section is Lemma \ref{l34}, that will play a crucial role in the proof of Theorem \ref{t11}. For this, we need to introduce technical terminologies. 

Let the circle act on a $2n$-dimensional compact oriented manifold $M$ with a discrete fixed point set. For each $p \in M$, denote by $\alpha_p^M$ the orientation on $T_pM$ given by the orientation on $M$. 

Let $w$ be a positive integer. As a subgroup of $S^1$, $\mathbb{Z}_w$ acts on $M$. Let $M^{\mathbb{Z}_w}$ be the set of points fixed by the $\mathbb{Z}_w$-action and $Z$ a connected component of $M^{\mathbb{Z}_w}$. Assume that $Z$ contains an $S^1$-fixed point $p$, i.e., $Z \cap M^{S^1} \neq \emptyset$. By Lemma \ref{l25}, $Z$ is orientable. Choose an orientation of $Z$. Since $M$ is oriented, then the normal bundle $NZ$ of $Z$ is orientable. Take an orientation on $NZ$ so that the orientation of $T_pZ \bigoplus N_pZ$ is the orientation of $T_pM$. Denote by $\alpha_p^N$ and $\alpha_p^Z$ the orientations on $N_pZ$ and $T_pZ$, respectively.

Let $p\in Z \cap M^{S^1}$ be an $S^1$-fixed point. As explained in the Introduction, the tangent space at $p$ decomposes into $n$ two-dimensional irreducible $S^1$-equivariant real vector spaces $L_1,\cdots,L_n$. Without loss of generality, by permuting the $L_i$'s, assume that $L_1,\cdots,L_m$ are the  summands of $N_pZ$ and $L_{m+1},\cdots,L_n$ are the summands of $T_pZ$, where $2(n-m)$ is the dimension of $Z$, i.e.,
\begin{center}
$T_pM=L_1 \bigoplus \cdots \bigoplus L_n$,

$N_pZ=L_1 \bigoplus \cdots \bigoplus L_m$, and

$T_pZ=L_{m+1} \bigoplus \cdots \bigoplus L_n$.
\end{center}
For each $i$, $L_i$ is isomorphic to a one-dimensional $S^1$-equivariant complex space, on which the action is given by multiplication by $g^{w_p^i}$, where $g \in S^1$ and $w_p^i$ is a non-zero integer. As in the Introduction, for each $i$, give an orientation on $L_i$ so that $w_p^i$ is positive. The choice of the orientation for each $L_i$ to make each weight $w_p^i$ positive gives orientation on $N_pZ$, $T_pZ$, and hence $T_pM$. Denote the orientations by $\beta_p^N$, $\beta_p^Z$, and $\beta_p^M$, respectively. 

Finally, define three parameters, depending on the two different orientations we made on each of $N_pZ$, $T_pZ$, and $T_pM$ in the following way:

\begin{Definition} \label{d31} 
\begin{enumerate}[(1)]
\item $\epsilon_p^N=+1$ if the two orientations $\alpha_p^N$ and $\beta_p^N$ on the normal bundle $N_pZ$ of $Z$ at $p$ agree and $\epsilon_p^N=-1$ otherwise.
\item $\epsilon_p^Z=+1$ if the two orientations $\alpha_p^Z$ and $\beta_p^Z$ on the tangent space $T_pZ$ of $Z$ at $p$ agree and $\epsilon_p^Z=-1$ otherwise.
\item $\epsilon_p^M=+1$ if the two orientations $\alpha_p^M$ and $\beta_p^M$ on the tangent space $T_pM$ of $M$ at $p$ agree and $\epsilon_p^M=-1$ otherwise.
\end{enumerate} 
\end{Definition}
By the definition, $\epsilon(p)=\epsilon_p^M$, where $\epsilon(p)$ is the sign of $p$ introduced in the Introduction. Moreover, it follows that $\epsilon_p^M=\epsilon_p^Z \cdot \epsilon_p^N$.

The next lemma states that if two $S^1$-fixed points $p$ and $q$ lie in the same connected component of $M^{\mathbb{Z}_w}$ for some positive integer $w$, then the weights at $p$ and the weights at $q$ have intimate relations.

\begin{lemma} \label{l31}
Let the circle act on a $2n$-dimensional compact oriented manifold $M$ with a discrete fixed point set. Fix a positive integer $w$. Let $\{p,q\} \subset Z \cap M^{S^1}$, where $Z$ is a connected component of $M^{\mathbb{Z}_w}$ such that $\dim Z=2n-2m$. Fix an orientation of $Z$. Rearrange weights at $p$ so that $\{w_p^1,\cdots,w_p^m\}$ are the weights on the normal bundle $N_pZ$ of $Z$ at $p$ and $\{w_p^{m+1},\cdots,w_p^n\}$ are the weights on the tangent space $T_pZ$ and similarly for $q$. Then there exist a permutation $\sigma \in S_m$ and $\epsilon \in \{-1,1\}^m$ such that $w_p^i \equiv \epsilon(i)w_q^{\sigma(i)} \mod w$ for any $1\leq i \leq m$ and $\epsilon_p^N=\epsilon_q^N \cdot (-1)^{\epsilon_{p,q}^-}$, where $\epsilon_{p,q}^-$ denotes the number of $i$'s such that $\epsilon(i)=-1$. \end{lemma}

\begin{proof} Since $TZ$ is oriented, $NZ$ is an oriented $\mathbb{Z}_w$-bundle over $Z$. Therefore, the $\mathbb{Z}_w$-representations of $N_pZ$ and $N_qZ$ are isomorphic, since $Z$ is connected. These representations are given by $\{\epsilon_p^N \cdot w_p^1, w_p^2, \cdots, w_p^m\} \mod w$ and $\{\epsilon_q^N \cdot w_q^1, w_q^2, \cdots, w_q^m\} \mod w$. Since they are isomorphic, there exist a permutation $\sigma \in S_m$ and $\epsilon \in \{-1,1\}^m$ such that $w_p^i \equiv \epsilon(i) w_q^{\sigma(i)} \mod w$ for each $i$. Moreover, $\epsilon_p^N=\epsilon_q^N \cdot (-1)^{\epsilon_{p,q}^-}$. \end{proof}

We explain Lemma \ref{l31} in an example.

\begin{exa} Let $S^1$ act on $\mathbb{CP}^4$ by
\begin{center}
$g \cdot [z_0:z_1:z_2:z_3:z_4]=[z_0:g z_1:g^2 z_2:g^3 z_3:g^4 z_4]$,
\end{center}
where $\mathbb{CP}^4$ is equipped with the standard orientation from the standard complex structure. The action has 5 fixed points, $p_0=[1:0:0:0:0]$, $p_1=[0:1:0:0:0]$, $p_2=[0:0:1:0:0]$, $p_3=[0:0:0:1:0]$, and $p_4=[0:0:0:0:1]$. Near $p_0$, $z_0 \neq 0$ and hence $\displaystyle \frac{z_i}{z_0}$ with $1 \leq i \leq 4$ are local coordinates near $p_0$. Therefore, the local action of $S^1$ near $p_0$ is given by
\begin{center}
$\displaystyle g \cdot (\frac{z_1}{z_0},\frac{z_2}{z_0},\frac{z_3}{z_0},\frac{z_4}{z_0})=(\frac{g \cdot z_1}{z_0},\frac{g^2 \cdot z_2}{z_0},\frac{g^3 \cdot z_3}{z_0},\frac{g^4 \cdot z_4}{z_0})=(g\cdot\frac{z_1}{z_0},g^2 \cdot  \frac{z_2}{z_0},g^3 \cdot \frac{z_3}{z_0},g^4 \cdot \frac{z_4}{z_0})$.
\end{center}
Therefore, the weights at $p_0$ as complex $S^1$-representations are $\{1,2,3,4\}$. Since all the weights are positive, as real $S^1$-representations, $\epsilon(p_0)=+1$ and the fixed point data at $p_0$ is $\{+,1,2,3,4\}$. Similarly, the weights at $p_3$ as complex $S^1$-representations are $\{-3,-2,-1,1\}$. Since there are 3 negative weights, as real $S^1$-representations, $\epsilon(p_3)=-1$ and the fixed point data at $p_3$ is $\{-,1,1,2,3\}$.

Now, the group $\mathbb{Z}_3$ acts on $\mathbb{CP}^4$ and one of the connected component $Z$ of the set $M^{\mathbb{Z}_3}$ of points that are fixed by the $\mathbb{Z}_3$-action is $\mathbb{CP}^1=S^2$, that contains $p_0$ and $p_3$. Equip $Z$ with the standard orientation from $\mathbb{CP}^4$. Since the weight of $T_{p_0}Z$ ($T_{p_3}Z$) as a complex representation is 3 (-3), we have $\epsilon_{p_0}^Z=+1$ ($\epsilon_{p_3}^Z=-1$, respectively). Since $\epsilon(p_0)=+1$ ($\epsilon(p_3)=-1$) and $\epsilon_{p_0}^Z=+1$ ($\epsilon_{p_3}^Z=-1$), and $\epsilon(p_0)=\epsilon_{p_0}^Z \cdot \epsilon_{p_0}^N$ ($\epsilon(p_3)=\epsilon_{p_3}^Z \cdot \epsilon_{p_3}^N$), we have that $\epsilon_{p_0}^N=+1$ ($\epsilon_{p_3}^N=+1$, respectively). Alternatively, since the weights at $N_{p_0}Z$ ($N_{p_3}Z$) as the complex representations are $\{1,2,4\}$ ($\{-2,-1,1\}$), we have $\epsilon_{p_0}^N=+1$ ($\epsilon_{p_3}^N=+1$, respectively). To sum up, we have

\begin{enumerate}
\item $\Sigma_{p_0}=\{+,1,2,3,4\}$, $\epsilon_{p_0}^Z=+1$, $\epsilon_{p_0}^N=+1$, and the weights of $N_{p_0}Z$ are $\{1,2,4\}$.
\item $\Sigma_{p_3}=\{-,1,1,2,3\}$, $\epsilon_{p_3}^Z=-1$, $\epsilon_{p_3}^N=+1$, and the weights of $N_{p_3}Z$ are $\{1,1,2\}$.
\end{enumerate}

The weights $\{1,2,4\}$ of $N_{p_0}Z$ and the weights $\{1,1,2\}$ of $N_{p_3}Z$ are equal modulo 3 up to sign;
\begin{center}
$1 \equiv 1 \mod 3, 2 \equiv -1 \mod 3$, and $4 \equiv -2 \mod 3$.
\end{center}
For the last two pairs, the weight at $p_0$ with plus sign is paired with the weight at $p_3$ with negative sign. Therefore, the terminology $\epsilon_{p_0,p_3}^-$ in Lemma \ref{l31} is equal to 2. By Lemma \ref{l31}, we must have that $1=\epsilon_{p_0}^N=\epsilon_{p_3}^N \cdot (-1)^{\epsilon_{p_0,p_3}^-}=1 \cdot (-1)^2=1$ and this computation confirms Lemma \ref{l31} in this example. \end{exa}

We shall discuss applications of Lemma \ref{l31}.

\begin{lemma} \label{l32}
Let the circle act on a compact oriented manifold $M$ with a discrete fixed point set. Let $w$ be a positive integer. Suppose that no multiples of $w$ occur as weights, other than $w$ itself. Then for any connected component $Z$ of $M^{\mathbb{Z}_w}$, the number of $S^1$-fixed points $p$ in $Z$ with $\epsilon_p^Z=+1$ and that with $\epsilon_p^Z=-1$ are equal. Moreover, we can pair points in $Z\cap M^{S^1}$ such that
\begin{enumerate}
\item If $(p,q)$ is a pair, then $\epsilon_p^Z=-\epsilon_q^Z$.
\item If $\{w_p^1,\cdots,w_p^m\}$ and $\{w_q^1,\cdots,w_q^m\}$ are the weights of $N_pZ$ and $N_qZ$, then there exist a bijection $\sigma_m:\{1,\cdots,m\}\rightarrow\{1,\cdots,m\}$ and $\epsilon \in \{-1,1\}^m$ such that $w_p^i \equiv \epsilon(i) w_q^{\sigma_m(i)} \mod w$ for any $1 \leq i \leq m$. Moreover, $\epsilon_p^M=\epsilon_q^M \cdot (-1)^{\epsilon_{p,q}^-+1}$, where $\epsilon_{p,q}^-$ is the number of $i$'s with $\epsilon(i)=-1$.
\end{enumerate}
\end{lemma}

\begin{proof} Consider the induced action of $S^1/\mathbb{Z}_w=S^1$ on $Z$. Since no multiples of $w$ occur as weights, every weight on the tangent space $T_pZ$ of $Z$ at a point $p$ that is fixed by the induced $S^1$-action is $1$. Apply Theorem \ref{t29} by taking $a_i=1$ to the induced $S^1$-action on $Z$. It follows that the number of fixed points $p$ with $\epsilon_p^Z=+1$ and that with $\epsilon_p^Z=-1$ are equal. Therefore, we can pair points in $Z\cap M^{S^1}$ so that if $(p,q)$ is a pair, then $\epsilon_p^Z=-\epsilon_q^Z$. By Lemma \ref{l31}, $\epsilon_p^N=\epsilon_q^N \cdot (-1)^{\epsilon_{p,q}^-}$. With $\epsilon_p^M=\epsilon_p^Z \cdot \epsilon_p^N$ and $\epsilon_q^M=\epsilon_q^Z \cdot \epsilon_q^N$, the lemma follows from Lemma \ref{l31}. \end{proof}

\begin{lemma} \label{l33}
Let the circle act effectively on a 4-dimensional compact oriented manifold $M$ with a discrete fixed point set. Suppose that a fixed point $p$ has weights $\{a,w\}$, where $w>1$. Then there exists a unique fixed point $q$ such that $\{p,q\}\subset S^2 \subset M^{\mathbb{Z}_w}$. Let $b$ be the remaining weight at $q$.
\begin{enumerate}
\item If $\epsilon(p)=\epsilon(q)$, then $a \equiv -b \mod w$.
\item If $\epsilon(p)=-\epsilon(q)$, then $a \equiv b \mod w$.
\end{enumerate}
\end{lemma}

\begin{proof} Since $w>1$ and the action is effective, if $Z$ is a connected component of $M^{\mathbb{Z}_w}$ that contains $p$, then $\dim Z=2$. The induced action of $S^1/\mathbb{Z}_w=S^1$ on $Z$ has $p$ as a fixed point. By Lemma \ref{l212}, it follows that $Z$ is the 2-sphere $S^2$ and it contains precisely two fixed points $p,q$. By applying Theorem \ref{t28} to the induced $S^1$-action on $Z$, we have $\epsilon_p^Z=-\epsilon_q^Z$. 

First, suppose that $\epsilon(p)=\epsilon(q)$. Then by Lemma \ref{l32}, since $\epsilon_p^M=\epsilon_q^M \cdot (-1)^{\epsilon_{p,q}^-+1}$, we have that $\epsilon_{p,q}^-=1$. Since $a$ and $b$ are the weights of the normal bundle of $Z$ at $p$ and $q$, respectively, by Lemma \ref{l32}, we have that $a \equiv -b \mod w$.

Second, suppose that $\epsilon(p)=-\epsilon(q)$. Then by Lemma \ref{l32}, since $\epsilon_p^M=\epsilon_q^M \cdot (-1)^{\epsilon_{p,q}^-+1}$, we have that $\epsilon_{p,q}^-=0$. Therefore, we have that $a \equiv b \mod w$. \end{proof}

Let us consider the biggest weight among the weights over all the fixed points. If it is strictly bigger than 1, then Lemma \ref{l33} has the following consequence that we will use in the proof of Theorem \ref{t11} and Theorem \ref{t12} in Section \ref{s6}.

\begin{lemma} \label{l34}
Let the circle act effectively on a 4-dimensional compact oriented manifold $M$ with a discrete fixed point set. Assume that the biggest weight $w$ is bigger than 1. If a fixed point $p$ has weight $w$, then there exists a unique fixed point $q$ such that $\{p,q\} \subset S^2 \subset M^{\mathbb{Z}_w}$. Moreover, if $a$ and $b$ are the remaining weights at $p$ and $q$, respectively, then the following holds:
\begin{enumerate}
\item If $\epsilon(p)=\epsilon(q)$, then $a+b=w$.
\item If $\epsilon(p)=-\epsilon(q)$, then $a=b$.
\end{enumerate}
\end{lemma}

\begin{proof} Since the action is effective and $w>1$ is the biggest weight, $a<w$ and $b<w$. Therefore, by Lemma \ref{l33}, the lemma follows. \end{proof}

\section{Multigraphs} \label{s4}

In this section, we associate a multigraph to a compact oriented manifold equipped with a circle action having a discrete fixed point set. In particular, we assign one with the following properties: each fixed point is a vertex, each vertex has exactly $n$-edges where $2n$ is the dimension of the manifold, each edge is indexed by a positive integer such that the labels of the edges at a vertex are the weights at the fixed point,  each vertex has a sign, and there is no self-loop. If a fixed point $p$ has weight $w$, then there exists another fixed point $q$ that has weight $w$. Therefore, we can draw an edge $e$ between $p$ and $q$ and assign a label $w$ to the edge $e$.

\begin{Definition} \label{d41} A \textbf{multigraph} $\Gamma$ is an ordered pair $\Gamma=(V,E)$ where $V$ is a set of vertices and $E$ is a multiset of unordered pairs of vertices, called \textbf{edges}. A multigraph is called \textbf{signed} if every vertex has sign $+$ or $-$. A multigraph is called \textbf{labelled}, if every edge $e$ is labelled by a positive integer $w(e)$, called the \textbf{label}, or the \textbf{weight} of the edge, i.e., there exists a map from $E$ to the set of positive integers. Let $\Gamma$ be a labelled multigraph. The \textbf{weights} at a vertex $v$ consists of labels (weights) $w(e)$ for each edge $e$ at $v$. A multigraph $\Gamma$ is called \textbf{$n$-regular}, if every vertex has $n$-edges. \end{Definition}

The following proposition shows that given a circle action on a compact oriented manifold with a discrete fixed point set, we can assign a signed, labelled multigraph that does not have any loops.

\begin{pro} \label{p42}
Let the circle act effectively on a $2n$-dimensional compact oriented manifold $M$ with a discrete fixed point set. Then there exists a signed, labelled, $n$-regular multigraph $\Gamma$ associated to $M$ with the following properties.
\begin{enumerate}
\item The set $V$ of vertices is the set $M^{S^1}$ of the fixed points.
\item The labels of the edges at a vertex $p$ are the weights at the corresponding fixed point $p$.
\item The sign $\epsilon(p)$ of a vertex $p$ is the sign $\epsilon(p)$ of the corresponding fixed point $p$.
\item If there is an edge $e$ between two vertices $p_1$ and $p_2$ with label $w$, then the corresponding fixed points $p_1$ and $p_2$ lie in the same connected component $Z$ of $M^{\mathbb{Z}_w}$.
\item The multigraph $\Gamma$ does not have any loops.
\end{enumerate} \end{pro}

\begin{proof} Assign each fixed point $p$ a vertex and also denote it by $p$. To each vertex $p$, assign a sign $\epsilon(p)$. These prove (1) and (3). Let $w$ be a positive integer and $Z$ a connected component of $M^{\mathbb{Z}_w}$, the set of points in $M$ that are fixed by the $\mathbb{Z}_w$-action. Assume that $Z$ contains an $S^1$-fixed point $p$. Then by Lemma \ref{l25}, $Z$ is orientable. Pick an orientation on $Z$. The $S^1$-action on $M$ restricts to an $S^1$-action on $Z$. Moreover, the smallest weight of the $S^1$-action on $Z$ is $w$. Applying Lemma \ref{l26} to the $S^1$-action on $Z$, it follows that the number of times the weight $w$ occurs at points that are fixed by the $S^1$-action on $Z$ with $\epsilon_p^Z=+1$ and that with $\epsilon_p^Z=-1$ are equal. Therefore, by this recipe, if a fixed point $p_1$ in $Z$ with $\epsilon_{p_1}^Z=+1$ has weight $w$ and a fixed point $p_2$ in $Z$ with $\epsilon_{p_2}^Z=-1$ has weight $w$, then we can draw an edge $e$ between the vertices $p_1$ and $p_2$, and issue a label $w$ to $e$. At each fixed point, we use one weight to draw only one edge. Repeat this for each positive integer $w$ and each connected component of the set $M^{\mathbb{Z}_w}$. This proves (2), (4), and (5), and that the multigraph is $n$-regular, hence the lemma follows. \end{proof}

\section{Blow-up} \label{s5}

Another key ingredient to prove Theorem \ref{t11} and Theorem \ref{t12} is a blow-up type operation. Blowing up the origin 0 in $\mathbb{C}^n$ is an operation that replaces the point 0 by the set of all straight complex lines through it. In this section, we shall introduce a blow-up type operation for an isolated fixed point of a circle action on a 4-dimensional oriented manifold by identifying a neighborhood of the fixed point with a neighborhood of the origin $0$ in $\mathbb{C}^2$. For this, let the circle act on a 4-dimensional oriented manifold $M$ and let $p$ be an isolated fixed point. Then we can identify a neighborhood $U$ of $p$ with a neighborhood $V$ of $0$ in $\mathbb{C}^2$, where the local action of $S^1$ near $p$ can be identified with a circle action near $0$ in $\mathbb{C}^2$ by
 \begin{center}
$g \cdot (z_1,z_2)=(g^{-a}z_1,g^b z_2)$,
\end{center}
for some positive integers $a$ and $b$ that are weights at $p$. We can choose the signs by reversing the orientation of $M$ if necessary. Now let us blow up $0$ in $\mathbb{C}^2$ in the usual sense. Since the neighborhood $U$ of $p$ is identified with the neighborhood $V$ of $0$ in $\mathbb{C}^2$, we make the corresponding change on $U$ of $p$. We shall call this procedure \textbf{blow up}, as what it does geometrically is the same as blow up in complex geometry. We blow up equivariantly so that there is a natural extended circle action on the blown up manifold. In this case, the equivariant blow up replaces the fixed point $p$ by $\mathbb{CP}^1=S^2$, and the extended $S^1$-action on the $\mathbb{CP}^1$ has two fixed points $p_1$ and $p_2$, whose fixed point data are $\{\epsilon(p),a,a+b\}$ and $\{\epsilon(p),b,a+b\}$, respectively. In other words, from $M$, by the blow up at $p$, we can construct a new manifold $M'$ equipped with a circle action whose fixed point data is the same as $M$ with replacing $\{\epsilon(p),a,b\}$ by $\{\epsilon(p),a,a+b\}\cup\{\epsilon(p),b,a+b\}$. To state our technical lemma, we need the following theorem, that is an application of the equivariant tubular neighborhood theorem (the slice theorem) to a fixed point of a group action on a manifold.

\begin{theo} (The Local Linearization Theorem) \cite{GGK} \label{t51} Let a compact Lie group $G$ act on a manifold $M$ and let $p \in M^G$ be a fixed point. Then there exists a $G$-equivariant diffeomorphism from a neighborhood of the origin in $T_pM$ onto a neighborhood of $p$ in $M$. \end{theo}

With Theorem \ref{t51}, we are ready to state our technical lemma, that allows us to blow up a fixed point of an $S^1$-action.

\begin{lemma} \label{l51}
Let the circle act on a 4-dimensional compact connected oriented manifold $M$ with a discrete fixed point set.
\begin{enumerate}[(1)]
\item Suppose that a fixed point $p$ has fixed point data $\{-,a,b\}$ for some positive integers $a$ and $b$. Then we can construct a 4-dimensional compact connected oriented manifold $\widetilde{M}$ equipped with a circle action such that the fixed point data of $\widetilde{M}$ is $(\Sigma_M \setminus \{-,a,b\}) \cup \{-,a,a+b\} \cup \{-,b,a+b\}$.
\item Suppose that a fixed point $p$ has fixed point data $\{+,a,b\}$ for some positive integers $a$ and $b$. Then we can construct a 4-dimensional compact connected oriented manifold $\widetilde{M}$ equipped with a circle action such that the fixed point data of $\widetilde{M}$ is $(\Sigma_M \setminus \{+,a,b\}) \cup \{+,a,a+b\} \cup \{+,b,a+b\}$.
\end{enumerate} 
\end{lemma}

\begin{proof} Assume Case (1), i.e., $\epsilon(p)=-1$. By Theorem \ref{t51}, a neighborhood of $p$ in $M$ is $S^1$-equivariantly diffeomorphic to a neighborhood of $0$ in $T_pM$. The tangent space at $p$ decomposes into 2 two-dimensional irreducible $S^1$-equivariant real vector spaces $L_1$ and $L_2$, on each of which the action is given by multiplication by $g^a$ and $g^b$ for any $g \in S^1$, respectively. Since $\epsilon(p)=-1$, this implies that $T_pM$ is isomorphic to $\mathbb{C}^2$, on which each $g \in S^1 \subset \mathbb{C}$ acts on $\mathbb{C}^2$ by either $g \cdot (z_1,z_2)=(g^{-a}z_1,g^bz_2)$ or $g \cdot (z_1,z_2)=(g^{a}z_1,g^{-b}z_2)$. By changing the orientations of both copies of $\mathbb{C}$'s in the latter case, we may assume that a neighborhood $U$ of $p$ is equivariantly diffeomorphic to a neighborhood $V$ of $0$ in $\mathbb{C}^2$ on which each $g \in S^1 \subset \mathbb{C}$ acts by
\begin{center}
$g \cdot (z_1,z_2)=(g^{-a}z_1,g^bz_2)$.
\end{center}
Denote the orientation preserving equivariant diffeomorphism by $\phi$. Next, we blow up the origin $0$ in $\mathbb{C}^2$ that corresponds to $p$. Call the blown up space $\widetilde{V}$. The blow up replaces 0 in $V$ by the set of all straight lines through it. The blown up space $\widetilde{V}$ is described as
\begin{center}
$\widetilde{V}=\{(z,l)|z \in l\}\subset V \times \mathbb{CP}^1$.
\end{center}
The space can also be described by the equation
\begin{center}
$\widetilde{V}=\{((z_1,z_2),[w_1:w_2])|(z_1,z_2) \in V, w_1z_2-w_2z_1=0\}$.
\end{center}
With this description, the action of $S^1$ on $V$ extends to act on $\widetilde{V}$ by
\begin{center}
$g \cdot ((z_1,z_2),[w_1:w_2])=((g^{-a}z_1,g^b z_2), [g^{-a}w_1:g^b w_2])$.
\end{center}
The extended action of $S^1$ on $\widetilde{V}$ has two fixed points $p_1=((0,0),[1:0])$ and $p_2=((0,0),[0:1])$. 

Now, we compute the fixed point data at $p_1$ and $p_2$. Consider $p_1$. Let $u=\frac{w_2}{w_1}$. Since $z_2=z_1\frac{w_2}{w_1}=z_1 u$, $z_1$ and $u$ become local coordinates near $p_1=((0,0),[1:0])$. Now, the extended action of $S^1$ near $p_1$ is given by
\begin{center}
$\displaystyle g\cdot(z_1,u)=g\cdot(z_1,\frac{w_2}{w_1})=(g^{-a}z_1,\frac{g^bw_2}{g^{-a}w_1})=(g^{-a}z_1,g^{a+b}u)$.
\end{center}
Hence, as complex $S^1$-representations, the weights at $p_1$ are $\{-a,a+b\}$. As real $S^1$-representations, the local fixed point data at $p_1$ is therefore $\{-,a,a+b\}$. Next, consider $p_2$. Let $v=\frac{w_1}{w_2}$. In this case $z_2$ and $v$ become local coordinates near $p_2$. The extended action of $S^1$ near $p_2$ is given by
\begin{center}
$\displaystyle g\cdot(z_2,v)=g\cdot(z_2,\frac{w_1}{w_2})=(g^bz_2,\frac{g^{-a}w_1}{g^bw_2})=(g^bz_2,g^{-a-b}v)$.
\end{center}
The weights at $p_2$ as complex $S^1$-representations are $\{b,-a-b\}$ and hence as real $S^1$-representations, the local fixed point data at $p_2$ is $\{-,b,a+b\}$. 

Note that $\phi$ is a diffeomorphism from $U \setminus \{p\}$ to $\widetilde{V} \setminus E$, where $E$ is the exceptional divisor. Consider the manifold $\widetilde{M}=((M\setminus\{p\})\sqcup \widetilde{V})/\phi$. The extended action of $S^1$ on the manifold $\widetilde{M}$ has the fixed point set $(M^{S^1}\setminus \{p\}) \cup \{p_1,p_2\}$ and hence the fixed point data $(\Sigma_M \setminus \{-,a,b\}) \cup \{-,a,a+b\} \cup \{-,b,a+b\}$. This proves the first part.

Assume Case (2), i.e., $\epsilon(p)=1$. In this case, reverse the orientation of $M$. This reverses the sign of $\epsilon(q)$ for each fixed point $q$. In particular, $\epsilon(p)=-1$. Now, proceed as in the first case; we blow up the fixed point $p$ to replace the fixed point $p$ with $\mathbb{CP}^1$, on which we have an extended $S^1$-action that has two fixed points $p_1$ and $p_2$. The fixed points $p_1$ and $p_2$ have fixed point data $\{-,a,a+b\}$ and $\{-,b,a+b\}$, respectively. We reverse the orientation of $M$ back to its original orientation and this completes the proof. \end{proof}

\section{Proof of the main result: Theorem \ref{t11} and Theorem \ref{t12}} \label{s6}

We are ready to prove Theorem \ref{t11} and Theorem \ref{t12}.

\begin{proof} [Proof of Theorem \ref{t11} and Theorem \ref{t12}]
Associate to $M$ a signed, labelled, 2-regular multigraph $\Gamma$ without any loop as in Proposition \ref{p42}. We begin with the biggest label (weight) of an edge among all the edges of the multigraph $\Gamma$. Let $e$ be an edge whose weight $w$ is the biggest among all the weights of the edges of the multigraph $\Gamma$. Assume that $w>1$. By (4) of Proposition \ref{p42}, the vertices (fixed points) $p_1$ and $p_2$ of the edge $e$ lie in the same connected component $Z$ of $M^{\mathbb{Z}_w}$. Let $a$ and $b$ be the remaining weights at $p_1$ and $p_2$, respectively. By Lemma \ref{l34}, $Z=S^2$, and either
\begin{enumerate}
\item $\epsilon(p_1)=\epsilon(p_2)$ and $a+b=w$, or
\item $\epsilon(p_1)=-\epsilon(p_2)$ and $a=b<w$.
\end{enumerate}

First, suppose that Case (1) holds; $\epsilon(p_1)=\epsilon(p_2)$ and $a+b=w$, i.e., $\Sigma_{p_1}=\{\epsilon(p_1),a,a+b\}$ and $\Sigma_{p_2}=\{\epsilon(p_1),b,a+b\}$. In this case, if we can construct a 4-dimensional compact connected oriented manifold $M'$, equipped with a circle action whose fixed point data is $(\Sigma_M \setminus (\{\epsilon(p_1),a,a+b\} \cup \{\epsilon(p_2),b,a+b\})) \cup \{\epsilon(p_1),a,b\}$, then by Lemma \ref{l51}, we can blow up the fixed point whose fixed point data is $\{\epsilon(p_1),a,b\}$ to construct a 4-dimensional compact connected oriented manifold $M''$ equipped with a circle action whose fixed point data $\Sigma_{M''}$ is the same as $\Sigma_M$. Therefore, the classification of the fixed point data of $M$ reduces to the existence of a 4-dimensional compact connected oriented manifold $M'$ equipped with a circle action whose fixed point data is $(\Sigma_M \setminus (\{\epsilon(p_1),a,a+b\} \cup \{\epsilon(p_2),b,a+b\})) \cup \{\epsilon(p_1),a,b\}$. Figure \ref{fig1} describes that by the blow up operation as in Lemma \ref{l51} at the fixed point $p$ in $M'$, from $M'$ whose multigraph is $\Gamma'$, we obtain a manifold $M''$ whose multigraph is $\Gamma$, which is the same as $\Gamma$ associated to $M$.

\begin{figure}
\centering
\begin{subfigure}[b][6.5cm][s]{.4\textwidth}
\centering
\vfill
\begin{tikzpicture}[state/.style ={circle, draw}]
\node[state] (a) {};
\node[state] (b) [above right=of a] {$p_1$};
\node[state] (c) [above=of b] {$p_2$};
\node[state] (d) [above left=of c] {};
\path (a) edge node[right] {$a$} (b);
\path (b) edge node [left] {$w=a+b$} (c);
\path (c) edge node [left] {$b$} (d);
\end{tikzpicture}
\vfill
\caption{$\Gamma$}\label{fig1-1}
\end{subfigure}
\begin{subfigure}[b][6.5cm][s]{.4\textwidth}
\centering
\vfill
\begin{tikzpicture}[state/.style ={circle, draw}]
\node[state] (a) {};
\node[state] (b) [above right=of a] {$p$};
\node[state] (c) [above left=of b] {};
\path (a) edge node[right] {$a$} (b);
\path (b) edge node [right] {$b$} (c);
\end{tikzpicture}
\vfill
\caption{$\Gamma'$}\label{fig1-2}
\vspace{\baselineskip}
\end{subfigure}\qquad
\caption{Case (1)}\label{fig1}
\end{figure}

When $\epsilon(p_1)=\epsilon(p_2)=+1$ (i.e., when $\epsilon(p)=+1$), this corresponds to Step (2) in Theorem \ref{t11}. When $\epsilon(p_1)=\epsilon(p_2)=-1$ (i.e., when $\epsilon(p)=-1$), this corresponds to Step (3) in Theorem \ref{t11}.

Second, suppose that Case (2) holds; $\epsilon(p_1)=-\epsilon(p_2)$ and $a=b$, i.e., $\Sigma_{p_1}=\{\epsilon(p_1),a,w\}$ and $\Sigma_{p_2}=\{-\epsilon(p_1),a,w\}$. Then there are two possibilities:
\begin{enumerate}[(a)]
\item There exists one more edge $e'$ between $p_1$ and $p_2$ with label $a$ (Figure \ref{fig2}(A)).
\item There exist other fixed points $p_3$ and $p_4$ such that there is an edge $e_1$ with label $a$ between $p_1$ and $p_3$, and there is an edge $e_2$ with label $a$ between $p_2$ and $p_4$ (Figure \ref{fig2}(B)).
\end{enumerate}

Assume that Case (a) holds. In this case, if we can construct a 4-dimensional compact connected oriented $M'$ equipped with a circle action whose fixed point data is $\Sigma_M \setminus (\{\epsilon(p_1),a,w\} \cup \{\epsilon(p_2),a,w\})$, then by Lemma \ref{l23}, we can perform equivariant sum of $M'$ and a circle action on $S^4$ to construct a 4-dimensional compact connected oriented manifold $M''$ equipped with a circle action whose fixed point data is the same as $\Sigma_M$. Here, $S^1$ acts on $S^4$ by $g \cdot (z_1,z_2,x)=(g^a z_1,g^w z_2,x)$, where $g \in S^1 \subset \mathbb{C}$, $z_i \in \mathbb{C}$ and $x \in \mathbb{R}$. The fixed point data of the circle action on $S^4$ is $\{+,a,w\} \cup \{-,a,w\}$. Therefore, the classification of the fixed point data of $M$ reduces to the existence of a 4-dimensional compact connected oriented manifold $M'$ equipped with a circle action whose fixed point data is $\Sigma_M \setminus (\{\epsilon(p_1),a,w\} \cup \{\epsilon(p_2),a,w\})$. This corresponds to Step (1) in Theorem \ref{t11}.

Assume that Case (b) holds. Let $a_3$ and $a_4$ be the remaining weights at $p_3$ and $p_4$, respectively, i.e., $\Sigma_{p_3}=\{\epsilon(p_3),a,a_3\}$ and $\Sigma_{p_4}=\{\epsilon(p_4),a,a_4\}$. Since $p_1$ and $p_3$ are connected by the edge $e_1$ whose label is $a$, by (4) of Proposition \ref{p42}, $p_1$ and $p_3$ lie in the same connected component of $M^{\mathbb{Z}_a}$. Suppose that $a>1$. If we take $a$ here for the role of $w$ in Lemma \ref{l33} ($w$ here for $a$ in Lemma \ref{l33} and $a_3$ here for $b$ in Lemma \ref{l33}), then we have that $w \equiv -a_3 \mod a$ if $\epsilon(p_1) =\epsilon(p_3)$, and $w \equiv a_3 \mod a$ if $\epsilon(p_1) \neq \epsilon(p_3)$. Equivalently, $-\epsilon(p_1)w \equiv \epsilon(p_3) a_3 \mod a$. This relation is trivial when $a=1$. Hence it holds for any $a$. Similarly, repeating this argument for the edge $e_2$ between $p_2$ and $p_4$ whose label is also $a$, we have $-\epsilon(p_2)w \equiv \epsilon(p_4) a_4 \mod a$.

In Case (b), we redraw edges $e_1$ between $p_1$ and $p_3$ and $e_2$ between $p_2$ and $p_4$ in the following way. Instead of the edge $e_1$ with label $a$ between $p_1$ and $p_3$, there is an edge $e_1'$ with label $a$ between $p_1$ and $p_2$. Instead of the edge $e_2$ with label $a$ between $p_2$ and $p_4$, there is an edge $e_2'$ with label $a$ between $p_3$ and $p_4$. Except redrawing of the two edges, other edges of $\Gamma$ remain the same. See Figure \ref{fig2}(B) for the multigraph $\Gamma$ and Figure \ref{fig2}(C) for the new multigraph $\Gamma'$.

Since our purpose is to classify the fixed point data, changing edges, i.e., isotropy submanifolds does not yield any problem. On the other hand, the new multigraph $\Gamma'$ still must satisfy Lemma \ref{l33} to be realized as a multigraph associated to a 4-dimensional compact connected oriented manifold equipped with a circle action having a discrete fixed point set, as the original multigraph $\Gamma$ satisfies. Now,  in the new multigraph $\Gamma'$, $p_1$ and $p_2$ are connected by the two edges $e$ with label $w$ and $e_1'$ with label $a$. Since $\epsilon(p_1)=-\epsilon(p_2)$, if we apply Lemma \ref{l33} for the edge $e$ whose label is $w$, then we have $a \equiv a \mod w$. If we apply Lemma \ref{l33} for the edge $e_1'$ whose label is $a$, we have $w \equiv w \mod a$. Therefore, Lemma \ref{l33} holds for the edges $e$ and $e_1'$.

Previously, since there were the edge $e_1$ with label $a$ between $p_1$ and $p_3$ and the edge $e_2$ with label $a$ between $p_2$ and $p_4$, by Lemma \ref{l33}, we had $-\epsilon(p_1)w \equiv \epsilon(p_3) a_3 \mod a$ and $-\epsilon(p_2)w \equiv \epsilon(p_4) a_4 \mod a$, respectively. Therefore, we had $-\epsilon(p_3) a_3 \equiv \epsilon(p_4) a_4 \mod a$, since $\epsilon(p_1)=-\epsilon(p_2)$. Now, by Lemma \ref{l33} for the edge $e_2'$ between $p_3$ and $p_4$ whose label is $a$ in the new multigraph $\Gamma'$,we must have $-\epsilon(p_3) a_3 \equiv \epsilon(p_4) a_4 \mod a$ and this confirms that Lemma \ref{l33} holds for the new multigraph $\Gamma'$.

With the redrawing of the edges, now we fall into the situation of Case (a). Therefore, proceed as in Case (a). Therefore, Case (b) also corresponds to Step (1) in Theorem \ref{t11}.

In the three Cases (1), (2)(a), and (2)(b) above, the classification problem of the fixed point data of $M$ reduces to the existence of a 4-dimensional compact connected oriented manifold $M'$ equipped with a circle action that has fewer fixed points. Repeat the process above whenever the biggest weight $w$ of an edge among all the edges is bigger than 1. The problem then reduces to the existence of a 4-dimensional compact connected oriented manifold $M'''$ equipped with a circle action, whose weights in the fixed point data are all 1, i.e., a semi-free circle action. By Theorem \ref{t29}, the number $k$ of fixed points $p$ in $M'''$ with $\epsilon(p)=+1$ and that with $\epsilon(p)=-1$ are equal. Such a manifold $M'''$ can be constructed as an eqivariant sum (along free orbits, in the sense of Lemma \ref{l23}) of $k$-copies of $S^4$'s, each of which is equipped with a circle action $g \cdot (z_1,z_2,x)=(g z_1,g z_2,x)$. The fixed point data of each $S^4$ is $\{+,1,1\} \cup \{-,1,1\}$. This also corresponds to Step (1) in Theorem \ref{t11}. \end{proof}

\begin{figure}
\begin{subfigure}[b][3.5cm][s]{.17\textwidth}
\centering
\vfill
\begin{tikzpicture}[state/.style ={circle, draw}]
\node[state] (a) {$p_1$};
\node[state] (b) [above=of a] {$p_2$};
\path (a) [bend right =20]edge node[right] {$w$} (b);
\path (b) [bend right =20]edge node[left] {$a$} (a);
\end{tikzpicture}
\vfill
\caption{Case(2)(a)}\label{fig2-1}
\vspace{\baselineskip}
\end{subfigure}\qquad
\begin{subfigure}[b][3.5cm][s]{.3\textwidth}
\centering
\vfill
\begin{tikzpicture}[state/.style ={circle, draw}]
\node[state] (a) {$p_1$};
\node[state] (b) [right=of a] {$p_3$};
\node[state] (c) [above=of a] {$p_2$};
\node[state] (d) [right=of c] {$p_4$};
\node[state] (e) [right=of b] {};
\node[state] (f) [right=of d] {};
\path (a) edge node[above] {$a$} (b);
\path (c) edge node[left] {$w$} (a);
\path (d) edge node[above] {$a$} (c);
\path (b) edge node[above] {$a_3$} (e);
\path (d) edge node[above] {$a_4$} (f);
\end{tikzpicture}
\vfill
\caption{Case(2)(b)$\Gamma$}\label{fig2-2}
\vspace{\baselineskip}
\end{subfigure}\qquad
\begin{subfigure}[b][3.5cm][s]{.3\textwidth}
\centering
\vfill
\begin{tikzpicture}[state/.style ={circle, draw}]
\node[state] (a) {$p_1$};
\node[state] (c) [right=of a] {$p_3$};
\node[state] (b) [above=of a] {$p_2$};
\node[state] (d) [above=of c] {$p_4$};
\node[state] (e) [right=of c] {};
\node[state] (f) [right=of d] {};
\path (a) [bend right =20]edge node[right] {$w$} (b);
\path (b) [bend right =20]edge node[left] {$a$} (a);
\path (c) edge node [right] {$a$} (d);
\path (c) edge node[above] {$a_3$} (e);
\path (d) edge node[above] {$a_4$} (f);
\end{tikzpicture}
\vfill
\caption{Case(2)(b)$\Gamma'$}\label{fig2-3}
\vspace{\baselineskip}
\end{subfigure}\qquad
\caption{Case (2)}\label{fig2}
\end{figure}

\section{$S^1$-actions on 4-manifolds with few fixed points and proof of Corollary \ref{c13}} \label{s7}

In this section, from Theorem \ref{t11}, we discuss the classification of $S^1$-actions on 4-dimensional compact oriented manifolds with few fixed points. From Theorem \ref{t11}, it is straightforward to classify the fixed point data when there are few fixed points. The case of two fixed points is given in Theorem \ref{t28} for any dimension.

\begin{theorem} \label{t71} Let the circle act on a 4-dimensional compact connected oriented manifold $M$ with 3 fixed points. Then the fixed point data of $M$ is the same as a blow-up at a fixed point of a rotation on $S^4$ with two fixed points, i.e., the fixed point data of $M$ is $\{\mp,a,b\}$, $\{\pm,a,a+b\}$, $\{\pm,b,a+b\}$ for some positive integers $a$ and $b$ with $\mathrm{sign}(M)=\pm 1$. \end{theorem}

\begin{proof} The fixed point data of $M$ is achieved by the combinatorics in Theorem \ref{t11}. Since there are 3 fixed points, Step (1) must occur precisely once. Then we have $\{+,a,b\} \cup \{-,a,b\}$ for some positive integers $a$ and $b$. If Step (1) occurs more than once then there are at least 4 fixed points. Next, to have 3 fixed points, exactly one of Step (2) and Step (3) must occur exactly once. If Step (2) occurs then it replaces $\{+,a,b\}$ by $\{+,a,a+b\} \cup \{+,b,a+b\}$ and hence the fixed point data of $M$ is $\{+,a,a+b\} \cup \{+,b,a+b\} \cup \{-,a,b\}$. Similarly, if Step (3) occurs then the fixed point data of $M$ is $\{+,a,b\} \cup \{-,a,a+b\} \cup \{-,b,a+b\}$. The conclusion on the signature of $M$ follows immediately from equation \ref{eq:1}. \end{proof}

\begin{theorem} \label{t72}
Let the circle act on a 4-dimensional compact connected oriented manifold $M$ with 4 fixed points. Then precisely one of the following holds for the fixed point data $\Sigma_M$ of $M$.
\begin{enumerate}
\item $\Sigma_M=\{+,a,b\} \cup \{-,a,b\} \cup \{+,c,d\} \cup \{-,c,d\}$ for some positive integers $a,b,c$, and $d$.
\item $\Sigma_M=\pm (\{-,a,b\} \cup \{+,a,a+b\} \cup \{+,b,a+2b\} \cup \{+,a+b,a+2b\})$ for some positive integers $a$ and $b$. 
\end{enumerate}
Moreover, in Case (1) $\textrm{sign}(M)=0$ and in Case (2) $\textrm{sign}(M)=\pm 2$. \end{theorem}

\begin{proof} The proof of the theorem is similar to that of Theorem \ref{t71}.  The fixed point data of $M$ is achieved by the combinatorics in Theorem \ref{t11} and Step (1) in Theorem \ref{t11} must occur at least once. If Step (1) occurs twice, then the fixed point data of $M$ is $\Sigma_M=\{+,a,b\} \cup \{-,a,b\} \cup \{+,c,d\} \cup \{-,c,d\}$ for some positive integers $a,b,c$, and $d$ and this is Case (1) of the theorem.

Suppose that Step (1) occurs exactly once. Assume that we apply Step (2) in Theorem \ref{t11}. Then we have  $\{+,a,a+b\} \cup \{+,b,a+b\} \cup \{-,a,b\}$ for some positive integers $a$ and $b$. If we apply Step (2) for $\{+,a,a+b\}$ or $\{+,b,a+b\}$, then this is Case (2) of the theorem. If we apply Step (3) for $\{-,a,b\}$, then the fixed point data of $M$ is $\{+,a,a+b\} \cup \{+,b,a+b\} \cup \{-,a,a+b\} \cup \{-,b,a+b\}$ and this is Case (1) of the theorem. The case where we apply Step (3) is similar. In each case, the conclusion on the signature of $M$ follows immediately from equation \ref{eq:1}. \end{proof}

It has not been known before if a manifold with fixed point data as in Case (2) in Theorem \ref{t72} exists or not, as the fixed point data for the case of 4 fixed points has not been classified. Our Theorem \ref{t12} proves the existence of a manifold of Case (2) in Theorem \ref{t72}. Note that the fixed point data of Case (2) in Theorem \ref{t72} cannot be realized as the fixed point data of a circle action on a complex manifold or a symplectic manifold, since with 4 fixed points in either case the signature of a manifold must vanish \cite{Jan3}.

We end this section with a proof of Corollary \ref{c13}.

\begin{proof} [Proof of Corollary \ref{c13}]
Equation \ref{eq:1} states that 
\begin{center}
$\displaystyle \textrm{sign}(M)=\sum_{p \in M^{S^1}} \epsilon(p)$.
\end{center}
Since $\epsilon(p)$ is either $+1$ or $-1$, it follows that $\textrm{sign}(M) \equiv k \mod 2$. Since there are finitely many fixed points, Step (1) in Theorem \ref{t11} must occur at least once. Moreover, Step (1), Step (2), and Step (3) in Theorem \ref{t11} only increase the number of fixed points with sign $+1$ and/or $-1$ and do not decrease. Therefore, not all fixed points can have $\epsilon(p)=+1$ ($\epsilon(p)=-1$). Hence we have $\textrm{sign}(M) \leq k-2$ ($\textrm{sign}(M) \geq 2-k$). This proves the first part.

For the second part, we first prove that for any integer $k \geq 2$, there exists a 4-dimensional compact connected oriented $S^1$-manifold $M$ with $k$ fixed points such that $\textrm{sign}(M)=k-2$ ($\textrm{sign}(M)=2-k$); $k-1$ fixed points have sign $+1$ ($-1$) and one fixed point has sign $-1$ ($+1$, respectively). We prove by induction on $k$. When $k=2$, a rotation on $S^4$ given by $g \cdot (z_1,z_2,x)=(g z_1, g z_2,x)$ provides the base case; it has 2 fixed points, its fixed point data is $\{+,1,1\}\cup\{-,1,1\}$, and hence $\textrm{sign}(S^4)=0$.

Assume that the claim holds for $k=i+2$. We prove for $i+3$. Take any 4-dimensional compact connected oriented $S^1$-manifold $M'$ with $i+2$ fixed points such that $\textrm{sign}(M')=i$ ($\textrm{sign}(M')=-i$). Take a fixed point $p$ whose sign is $+1$ ($-1$); such a fixed point $p$ exists by Theorem \ref{t11}. Let the fixed point data at $p$ be $\Sigma_p=\{+,a,b\}$ ($\Sigma_p=\{-,a,b\}$) for some positive integers $a$ and $b$. By Lemma \ref{l51}, from $M'$, by blowing up the fixed point $p$ we can construct a 4-dimensional compact connected oriented manifold $\widetilde{M'}$ equipped with a circle action such that the fixed point data of $\widetilde{M'}$ is $(\Sigma_{M'} \setminus \{+,a,b\}) \cup \{+,a,a+b\} \cup \{+,b,a+b\}$ ($(\Sigma_{M'} \setminus \{-,a,b\}) \cup \{-,a,a+b\} \cup \{-,b,a+b\}$). We create one more fixed point with sign $+1$ ($-1$). Therefore, the manifold $\widetilde{M'}$ has $i+3$ fixed points; $i+2$ fixed points have sign $+1$ ($-1$) and one fixed point has sign $-1$ ($+1$). Since $\textrm{sign}(\widetilde{M'})=\sum_{p \in \widetilde{M'}^{S^1}} \epsilon(p)$, the signature of $\widetilde{M'}$ is $\textrm{sign}(\widetilde{M'})=i+1$ ($\textrm{sign}(\widetilde{M'})=-i-1$). This proves the claim.

Suppose that we are given a pair $(j,k)$ of integers such that $k \geq 2$, $2-k \leq j \leq k-2$, and $j \equiv k \mod 2$. By the claim above, there exists a 4-dimensional compact connected oriented $S^1$-manifold $M'$ with $|j|+2$ fixed points such that $\textrm{sign}(M')=j$ ($-j$). Take a connected sum of $M'$ and $\frac{k-|j|-2}{2}$ ($k-|j|-2$ is even) copies of $S^4$'s, on each of which the action is given by $g \cdot (z_1,z_2,x)=(g z_1, g z_2,x)$ and hence the fixed point data $\{+,1,1\}$ and $\{-,1,1\}$. Here, the connected sum is along free orbits of the manifolds in the sense of Lemma \ref{l23}. The resulting manifold $M$ is a 4-dimensional compact connected oriented manifold and is equipped with a circle action that has $k$ fixed points. Since $M'$ has $|j|+1$ fixed points with sign $+1$ ($-1)$ and one fixed point with sign $-1$ ($+1$) and each rotation on $S^4$ has one fixed point with sign $+1$ and one fixed point with sign $-1$, the manifold $M$ has $|j|+1+\frac{k-|j|-2}{2}$ fixed points with sign $+1$ ($-1$) and $1+\frac{k-|j|-2}{2}$ fixed points with sign $-1$ ($+1$). Again, since $\textrm{sign}(M)=\sum_{p \in M^{S^1}} \epsilon(p)$, the signature of $M$ is $\textrm{sign}(M)=j$ ($-j$, respectively). \end{proof}

\begin{rem} For the second part of Corollary \ref{c13}, the construction of a manifold is not unique. For instance, suppose that we construct an example with 6 fixed points and with signature 2. Then one may take two copies of a manifold with signature 1 in Theorem \ref{t71} and take a connected sum in the sense of Lemma \ref{l23}, or take blowing up twice at fixed points with sign $+1$ of a rotation on $S^4$ and take a connected sum with another rotation on $S^4$. The fixed point data of the former manifold is $\{-,a,b\} \cup \{+,a,a+b\} \cup \{+,b,a+b\} \cup \{-,c,d\} \cup \{+,c,c+d\} \cup \{+,d,c+d\}$ for some positive integers $a,b,c$, and $d$, whereas for the latter it is $\{-,e,f\} \cup \{+,e,e+f\} \cup \{+,f,e+2f\} \cup \{+,e+f,e+2f\} \cup \{+,g,h\} \cup \{-,g,h\}$ for some positive integers $e,f,g$, and $h$. \end{rem}

\section{Graphs as manifolds} \label{s8}

In section \ref{s4}, we associated a signed, labelled multigraph without a loop, to a compact oriented manifold equipped with a circle action, having a discrete fixed point set. The multigraph is $n$-regular, where $2n$ is the dimension of the manifold. Moreover, it satisfies equal modulo property for weights of edges in the sense of Lemma \ref{l31}. Now, we may ask a question: under what conditions, does a multigraph behave like a manifold? That is, when can a multigraph be realized as a multigraph associated to a manifold equipped with a circle action? We answer this question when a multigraph is 2-regular.

\begin{Definition} \label{d81}
Let $\Gamma$ be a signed, labelled, 2-regular multigraph that does not have any loop.
\begin{enumerate}
\item The multigraph $\Gamma$ is called \textbf{effective}, if for every vertex $v$, its edges $e_1$ and $e_2$ have relatively prime labels.
\item The multigraph $\Gamma$ is said to satisfy \textbf{equal modulo property}, if for any edge $e$ between two vertices $v_1$ and $v_2$, $e_i$ is the remaining edge at $v_i$, then $-\epsilon(v_1) w(e_1) \equiv \epsilon(v_2)w(e_2) \mod w(e)$.
\item The multigraph $\Gamma$ is said to satisfy \textbf{minimal property}, if an edge $e$ has the smallest weight (label), then its vertices $v_1$ and $v_2$ satisfy $\epsilon(v_1)=-\epsilon(v_2)$.
\end{enumerate} \end{Definition}

With Definition \ref{d81}, we provide a necessary and sufficient condition for a 2-regular multigraph to be realized as a multigraph associated to a circle action on an oriented 4-manifold.

\begin{theorem} \label{t82} Let $\Gamma$ be an effective, signed, labelled, 2-regular multigraph that does not have any loops. Suppose that $\Gamma$ satisfies the equal modulo property and the minimal property. Then there exists a 4-dimensional compact connected oriented manifold $M$ equipped with an effective circle action having a discrete fixed point set such that the corresponding multigraph is $\Gamma$. \end{theorem}

\begin{proof} The proof is analogous to the proof of Theorem \ref{t11} and Theorem \ref{t12} in Section \ref{s6}. Let $e$ be an edge whose label $w(e)$ is the largest among all the labels of edges. Let $v_1$ and $v_2$ be its vertices. Suppose that $w(e)>1$. If $\epsilon(v_1)=\epsilon(v_2)$, then by the equal modulo property of $\Gamma$, the remaining edges $e_i$ at $v_i$ satisfy $w(e_1)+w(e_2)=w(e)$. As in Case (1) of the proof of Theorem \ref{t11}, the existence of $M$ whose multigraph is $\Gamma$ reduces to the existence of a manifold $M'$ whose multigraph $\Gamma'$ (Figure \ref{fig1-2}) is the same as $\Gamma$ (Figure \ref{fig1-1}) with the edge $e$ shrunk to a vertex (Figure \ref{fig1}).

Similarly, in the case that $\epsilon(v_1) \neq \epsilon(v_2)$, we proceed as in Case (2) of the proof of Theorem \ref{t11} in Section \ref{s6}; the existence of $M$ with multigraph $\Gamma$ reduces to the existence of a manifold $M'$ with $\Gamma'$ that has fewer vertices than $\Gamma$.

At each step, the minimal property is conserved. Finally, after a finite number of the steps as above, the label of every edge is 1. Such a multigraph is realized as a multigraph associated to an equivariant sum by means of Lemma \ref{l23}, of copies of $S^4$'s on each of which the action is given by $g \cdot (z_1,z_2,x)=(gz_1,gz_2,x)$. \end{proof}

As an application of Theorem \ref{t82}, we give a necessary and sufficient condition for a fixed point data to be realized as the fixed point data of a circle action on an oriented 4-manifold that does exist.

\begin{theorem} \label{t83} Let $\Sigma=\cup_{p} \{\epsilon(p),w_p^1,w_p^2\}$ be a finite collection, where $\epsilon(p)=\pm 1$ and $w_p^i$ are positive integers. Then there exists a 4-dimensional compact connected oriented manifold $M$ equipped with an effective circle action whose fixed point data is $\Sigma$, if and only if there exists a multigraph $\Gamma=(V,E)$ such that
$\cup_{v \in V} \{\epsilon(v),w(e_1(v)),w(e_2(v))\}=\Sigma$, and $\Gamma$ satisfies the conditions in Theorem \ref{t82}. \end{theorem}

\end{document}